\documentclass{siamart1116}
\usepackage{graphicx}
\usepackage{epstopdf}
\usepackage{amsmath,amsfonts}
\usepackage{color}
\ifpdf
  \DeclareGraphicsExtensions{.eps,.pdf,.png,.jpg}
\else
  \DeclareGraphicsExtensions{.eps}
\fi

\numberwithin{theorem}{section}
\newcommand{\TheTitle}{An $L1$ approximation for a fractional reaction-diffusion equation,  a second-order error analysis over time-graded meshes}
\newcommand{\TheAuthors}{Kassem Mustapha}
\headers{An L1 method for a fractional diffusion model}{\TheAuthors}
\title{{\TheTitle}\thanks{{This work was supported by the KFUPM.}
}}

\author{
Kassem Mustapha\thanks{Department of Mathematics and Statistics,
King Fahd University of Petroleum and Minerals,
Dhahran 31261, Saudi Arabia (\email{kassem@kfupm.edu.sa}).}
}

\ifpdf
\hypersetup{
  pdftitle={\TheTitle},
  pdfauthor={\TheAuthors}
}
\fi
\newsiamremark{remark}{Remark}
\newcommand{\K}{\mathbf{L}}
\newcommand{\Ba}{\partial_t^{1-\alpha}}
\newcommand{\R}{\mathbb{R}}
\newcommand{\fInt}{\mathcal{I}}
\newcommand{\iprod}[1]{\langle#1\rangle}
\newcommand{\bigiprod}[1]{\biggl\langle#1\biggr\rangle}

\newcommand{\matM}{\boldsymbol{M}}
\newcommand{\W}{\mathcal W}
\newcommand{\matG}{\boldsymbol{G}}

\newcommand{\A}{\mathcal{A}}
\begin{document}

\date{\today}

\maketitle

\begin{abstract}
A time-stepping $L1$ scheme for subdiffusion equation with a Riemann--Liouville time-fractional derivative is developed and analyzed. This is the first paper to show that the $L1$ scheme for the model problem under consideration  is second-order accurate (sharp error estimate) over nonuniform time-steps. The established convergence analysis is novel, innovative and concise. For completeness, the $L1$ scheme is combined with the standard Galerkin finite elements for the spatial discretization, which will then define a fully-discrete numerical scheme. The error analysis for this scheme is also investigated.  To support our theoretical contributions, some numerical tests are provided at the end. The considered (typical)  numerical example  suggests that the imposed time-graded meshes assumption can be further relaxed. 
\end{abstract}

%%%%%%%%%%%%%%%%%%%%%%%%%%%%%%%%%%%%%%%%%%%%%%%%%%%%%%%%%%%%%%%%%%%%%%
\section{Introduction}
Consider the following  time-fractional diffusion equation,
\begin{equation}\label{eq: ibvp}
\partial_t u(x,t)
	+\partial_t^{1-\alpha}\A u(x,t)
	=f(x,t),
	\quad\text{for $x\in\Omega$ and $0<t<T$,}
\end{equation}
 with  initial condition $u(x,0)=u_0(x)$,
where $\partial_t=\partial/\partial t$, $\Omega$ is a convex polyhedral
domain in~$\R^d$ ($d\ge1$), and the spatial elliptic operator 
\[\A u(x,t)=-\nabla\cdot(\kappa_\alpha(x)\nabla u(x,t))+d(x)u(x,t)\,.\]
 The diffusivity
coefficient $c_1\le \kappa_\alpha\le c_2$ on $\Omega$ for some positive constants $c_1$ and $c_2,$ and the reaction coefficient { $d$ is such that the bilinear form associated with the elliptic operator $\A$ (see \eqref{eq: bilinear}) is positive definite  on the Sobolev space $H_0^1 (\Omega)$. That is, it is sufficient (but not necessary) to impose that $d \ge 0$ on $\Omega$.} Both, $\kappa$ and $d$ are assumed to be sufficiently regular functions.  

The fractional exponent is restricted
to the range~$0<\alpha<1$ and  the fractional derivative is
taken in the Riemann--Liouville sense, that is,
$\partial_t^{1-\alpha}u=\partial_t\fInt^\alpha u$, where the
fractional integration operator~$\fInt^\alpha$ is defined by
\[
\fInt^\alpha v(t)
	=\int_0^t\omega_\alpha(t-s)v(s)\,ds,
	\quad\omega_\alpha(t)=\frac{t^{\alpha-1}}{\Gamma(\alpha)}.
\]
We impose  a homogeneous Dirichlet boundary condition,
\begin{equation}\label{eq: Dirichlet bc}
u(x,t)=0\quad\text{for $x\in\partial\Omega$ and $0<t<T$.}
\end{equation}
Over the last decade, various time-stepping numerical methods were invesitigated for solving the fractional diffusion equation \eqref{eq: ibvp}, see for example \cite{Mustapha2011,Mustapha2015} and related refreences therein. The motivation of this paper is propose and analyze a second-order accurate time-stepping $L1$ scheme for solving the model problem \eqref{eq: ibvp}.  A nonuniform time mesh is employed (see \eqref{eq: time mesh}) to compensate for the singularity of the continuous solution near $t=0$ \cite{McLean2010,McLeanMustaphaAliKnio2019}. {Such graded meshes \eqref{eq: time mesh} were originally used in the context of Volterra integral equations with weakly singular kernels, see for example \cite{Brunner2004, BrunnerPedasVainikko1999, ChandlerGraham1988}, and see also \cite{Mustapha2013} for a recent concrete superconvergence error analysis. Later on, time-graded meshes were successfully used to improve the performance of different numerical methods applied to fractional diffusion and fractional wave equations, see for instance \cite{McLeanMustapha2007, McLeanThomeeWahlbin1996, Mustapha2011, MustaphaAbdallahFursti2014}}. Nonuniform meshes are flexible and reasonably convenient for practical implementation, however they can significantly complicate the numerical error analysis of schemes. { 
%The challenging part in showing optimal-order convergence results is in the clever use of the time-graded mesh properties
The time-graded mesh properties are carefully  used in our error analysis  to achieve  optimal-order convergence rates.}   The designed  approach   is novel and concise, some innovative ideas are employed  to estimate efficiently certain candidates. For completeness, we discretize in space using the standard Galerkin finite elements, where  the error analysis  is also examined.   

To the best of our knowledge, we are not aware of any work that showed a  second-order error bounds  of the  popular  time-stepping $L1$ scheme applied to the  model problem \eqref{eq: ibvp}.  However, for the time-fractional (Caputo derivative) diffusion problem (often assuming  $\kappa$ to be constant and the reaction coefficient $d$ to be zero): 
\begin{equation}\label{eq: ibvp caputo}
\fInt^{1-\alpha}\partial_t u(x,t)
+	\A u(x,t)=f(x,t),
	\quad\text{for $x\in\Omega$ and $0<t<T$,}
\end{equation}
 various types of $L1$ time-stepping schemes were developed and studied  over the last decade, see for example \cite{Alikhanov2015,ChenLiuAnhTurner2010,GaoSun2011,JiangMa2011,JinLazarovZhou,LiaoLiZhang2018,LiaoMcLeanZhang2019,LinXu2007,
 StynesORiordanGracia2017,WangZhaoChenWeiTang2018,YanKhanFord2018,ZhaoChenBuLiuTang2017,ZhaoZhangShiLiuTurner2016}.
In most studies, a convergence rate of order $2-\alpha$ was proved. Furthermore, the singularity of the continuous solution $u$ near $t=0$  was taken into account in a few papers only, however the rest frequently ignored this fact. In contrast, a time-stepping discontinuous Petrov-Galerkin method using piecewise polynomials of degree $m$ was introduced and analyzed in \cite{MustaphaAbdallahFursti2014} { for solving  problem \eqref{eq: ibvp caputo}}. For the special case  $m=1$, this method reduces to a second-order accurate  time-stepping  $L1$ scheme { as the numerical results suggested therein, see \cite[Section 5]{MustaphaAbdallahFursti2014}.}

Outline of the paper. In section \ref{sec: scheme}, we define our semi-discrete time-stepping $L1$ approximation scheme (see \eqref{fully}) and describe briefly the implementation steps. Section \ref{sec: error analysis} is dedicated to show  our sharp error results. It is assumed that the continuous solution $u$  of problem \eqref{eq: ibvp} satisfies the following regularity properties:   
\begin{equation}\label{eq: regularity} 
 \|u(t)\|_2\le M\quad{\rm and}\quad \|u'(t)\|_2+t^{1-\alpha/2}\|u''(t)\|_1+t^{2-\alpha/2}\|u'''(t)\|_1
	\le M t^{\sigma-1},
\end{equation}
 for some  positive  constants $M$ and $\sigma$. In \eqref{eq: regularity}, $'$ denotes the time partial derivative and  $\|\cdot\|_\ell$ is the norm on the usual Sobolev space $H^\ell(\Omega)$ which reduces to the $L_2(\Omega)$-norm when $\ell=0$ denoted by $\|\cdot\|$.  As an example,  when $f(t)\equiv0$~and 
$u_0\in H^1_0(\Omega)\cap H^{2.5^-}(\Omega)$, these assumptions hold true for $\sigma=\frac{\alpha^-}{4}$, see \cite{McLean2010,McLeanMustaphaAliKnio2019} for more details. 

At each time level $t_n$, an optimal $O(\tau^2 t_n^{\sigma+\alpha-2/\gamma})$-rate of  convergence is proved in Theorem \ref{thm: time convergence}, assuming that the  time mesh exponent $\gamma > \max\{2/(\sigma+\alpha/2),2/(\sigma+3\alpha/2-1/2)\}$ (see \eqref{eq: time mesh} for the definition of the time-graded mesh). Noting that, for $1/2\le \alpha<1$ (which is practically the interesting case in terms of subdiffusion), $\sigma+\alpha/2\le \sigma+3\alpha/2-1/2$, and so it is sufficient to assume $\gamma > 2/(\sigma+\alpha/2)$. Our error analysis  involves various types of clever splitting of the error terms followed by a careful estimation of each one of them. 
We avoid using  { any versions of the weakly singular discrete Gronwall's inequalities \cite[Theorem 6.1]{DixonMcKee1986}} to guarantee that the error coefficients do not blowup exponentially with the time level $t_n.$ At the preliminary stage, our error analysis   makes use of   the inequality in the next lemma \cite[Lemma 2.3]{McLeanMustaphaAliKnio2019o} which will eventually enable us to establish pointwise estimates for certain terms.
\begin{lemma}\label{lem: pointwise bound}
Let $0<\alpha\le 1$. If the function 
$\phi$ is in the space $W^1_1\bigl((0,t);L_2(\Omega)\bigr)$ satisfies
$\phi(0)=\fInt^\alpha\phi'(0)=0$, then
\[\|\phi(t)\|^2\le 2\omega_{2-\alpha}(t)\int_0^{t}\iprod{\fInt^\alpha\phi'(t),\phi'(t)}\,dt,\]
\end{lemma}
  where $\iprod{u,v}$ is the $L_2$-inner product on the spatial domain $\Omega$.
 
Although the main scope of this paper is on the optimal error analysis of the time-stepping $L1$ scheme, the error analysis from the full discretization is also studied.   In section \ref{sec: fully-discrete}, the semi-discrete time-stepping scheme \eqref{fully} is discretized in space via the standard continuous piecewise-linear  Galerkin method (see \eqref{fully 2 2}), which will then define a fully-discrete numerical method. The implementation of the fully-discrete solution is briefly discussed. Compared to the error analysis  in section \ref{sec: error analysis}, an additional term has occurred. Consequently, an additional error of order $O(h^2)$ ($h$ is the maximum spatial mesh element size) is derived assuming that $\sigma>(1-\alpha)/2$, see Theorem \ref{thm: CR}. Numerically, it is observed that this condition is not necessary.  In this part of our error analysis,  the next lemma  \cite[Lemma 3.1]{MustaphaSchotzau2014} is used. 
\begin{lemma}\label{lem: alpha dep}
If the functions   $\phi$ and $\psi$  are in the space $L_2\bigl((0,t);L_2(\Omega)\bigr)$, then  for  $0<\alpha<1$ and for $\epsilon>0$,   
\[\biggl|\int_0^t\iprod{\phi,\fInt^\alpha\psi}\,ds\biggr|
	\le\frac{1}{4\epsilon(1-\alpha)^2}\int_0^t\iprod{\phi,\fInt^\alpha\phi}\,ds
	+\epsilon\, \int_0^t\iprod{\psi,\fInt^\alpha\psi}\,ds\,.\]
\end{lemma}
Unfortunately, the coefficient $\frac{1}{(1-\alpha)^2}$ blows up as $\alpha$ approaches $1^-$, and consequently, the error bounds blowup. { Such a blowup phenomenon, which was highlighted and investigated  recently in \cite{ChenStynes2019},  occurs in the error analysis (but not in numerical experiments \cite{McLeanMustapha2015}) of various numerical methods applied to different time-fractional diffusion models, see for example \cite{JinLiZhou2019, Karaa2018, KaraaMustaphaPani2018, Kopteva2019, LiaoLiZhang2018, McLeanMustapha2015, Mustapha2015, MustaphaAbdallahFursti2014, StynesORiordanGracia2017}.  This blowup behavior appears to be an artifact of the method of proof, see Remark \ref{remark: blows up} where the blowup coefficient is controlled  assuming that  $\sigma>(1-\alpha)/2$. To validate this, some compatibility conditions on the initial data $u_0$ (for example, $u_0 \in H^{2+1/\alpha}(\Omega)$ with $u_0,\A u_0\in H^1_0(\Omega)$) and also on the source term $f$ are needed. Noting that, in the limiting case, $\alpha \to 1^-$,  problem \eqref{eq: ibvp} reduces to the classical  equation \eqref{eq: CD} and our fully-discrete scheme in \eqref{fully 2 2} amounts to the sandard time-stepping Crank-Nicolson (see \eqref{eq: CN}) combined with the (linear) spatial standard continuous  Galerkin method. A straightforwrd analysis leads to an optimal time-space second-order convergence rate \cite{Thomee2006}.}

Finally, in section \ref{sec: numerical results},  a second-order convergence of the $L1$ scheme is confirmed numerically on a typical sample  of test problem. When the time error is dominant, the numerical numbers  in Tables \ref{table 1}--\ref{table 3}  illustrate $O(\tau^{\min\{\gamma(\sigma+\alpha),2\}})$-rates for different choices  of the time-graded mesh exponent $\gamma$ and the fractional exponent  $\alpha$.  These  results  indicate that the condition  $\gamma > \max\{2/(\sigma+\alpha/2),2/(\sigma+3\alpha/2-1/2)\}$ in Theorem \ref{thm: time convergence} is pessimistic. Practically, it is enough to choose $\gamma = 2/(\sigma+\alpha)$ to guarantee an $O(\tau^2)$  accuracy.  Furthermore, the numerical results in Table \ref{table 4} showed $O(h^2)$-rates of convergence in space for different values of $\alpha$ even though  the assumption $\sigma>(1-\alpha)/2$ is not satisfied. 

For later use,   $A(\cdot, \cdot):H_0^1 (\Omega)\times H_0^1(\Omega)\rightarrow\mathbb{R}$ denotes  the bilinear form associated with the elliptic operator $\A$, which is symmetric and positive definite, defined by 
 \begin{equation}\label{eq: bilinear} A(v, w)=\iprod{\kappa_\alpha\nabla v,\nabla w}+\iprod{d\,v,w}.\end{equation}
Throughout the paper,   $C$ is a generic constant which may depend on the parameters $M$, $\sigma$, $T$, $\Omega$,   and  $\gamma$, but is independent of $\tau$ and $h$.

%%%%%%%%%%%%%%%%%%%%%%%%%%%%%%%%%%%%%%%%%%%%%%%%%%%%%%%%%%%%%%%%%%%%%
\section{Numerical method}\label{sec: scheme}
This section is devoted to introduce our semi-discrete time-stepping $L1$ numerical scheme for solving the model problem \eqref{eq: ibvp}. We use a time-graded mesh with the following nodes:  
\begin{equation} \label{eq: time mesh} 
 t_i=(i\,\tau)^\gamma,~~~{\rm for}~~ 0\le i\le N,~~{\rm for}~~\gamma\ge 1, ~~{\rm with}~~\tau=T^{1/\gamma}/N,\end{equation}
 where $N$ is the number of subintervals. Denote by~$\tau_n=t_n-t_{n-1}$ the length of the $n$th
subinterval~$I_n=(t_{n-1},t_n)$, for $1\le n\le N$. It is not hard to show  that such a time-graded mesh has the following properties \cite{McLeanMustapha2007}: for $n\ge 2,$ 
\begin{equation}\label{eq: mesh property 1}
t_n\le 2^\gamma t_{n-1},\quad \gamma \tau t_{n-1}^{1-1/\gamma}\le \tau_n\le \gamma \tau t_n^{1-1/\gamma},\quad \tau_n-\tau_{n-1} \le C_\gamma \tau^2 \min(1,t_n^{1-2/\gamma})\,.
\end{equation}

For a given function $v$ defined on the time interval $[0,T],$ let $v^n=v(t_n)$ for $0\le n\le N.$ With this grid function, 
 we associate the backward   difference,
\[
\partial v^n=\frac{v^n-v^{n-1}}{\tau_n}.
\]
To define our time-stepping numerical scheme, integrating problem \eqref{eq: ibvp} over the time interval $I_n$, 
\begin{equation}\label{eq: ibvp int}
\int_{t_{n-1}}^{t_n}u'(t)\,dt +\int_{t_{n-1}}^{t_n}\partial_t^{1-\alpha}\A u(t)\,dt= \int_{t_{n-1}}^{t_n} f(t)\,dt\,.
\end{equation}
Our $L1$ approximate  solution $U$, which is a continuous linear polynomial in the time variable on each closed subinterval $[t_{n-1},t_n],$   is defined by replacing $u$ with $U$ in \eqref{eq: ibvp int}, 
\begin{equation} \label{fully}
 U^n-U^{n-1}+ \int_{t_{n-1}}^{t_n}\partial_t^{1-\alpha}\A U(t)\,dt= \int_{t_{n-1}}^{t_n} f(t)\,dt,~~{\rm for}~~~
 1\le n\le N,
\end{equation}
with $U^0=u_0.$ As $\alpha \to 1^-$, the fractional model problem \eqref{eq: ibvp} amounts to the classical reaction-diffusion equation:
\begin{equation} \label{eq: CD} 
 u'(x,t)
	+\A  u(x,t)=f(x,t),
	\quad\text{for $x\in\Omega$ and $0<t<T$,}
\end{equation}
and the time-stepping $L1$ numerical scheme  \eqref{fully} reduces to 
\begin{equation} \label{eq: CN} U^n-U^{n-1}+\tau_n \A(U^n+U^{n-1})/2= \int_{t_{n-1}}^{t_n} f(t)\,dt,\end{equation}
which is the time-stepping Crank-Nicolson method   for problem \eqref{eq: CD}.  Motivated by this, a  generalized Crank-Nicolson scheme for the fractional reaction-diffusion equation  \eqref{eq: ibvp}, defined by 
\begin{equation} \label{GCN}
 U^n-U^{n-1}+ \int_{t_{n-1}}^{t_n}\partial_t^{1-\alpha}\A \overline U(t)\,dt= \int_{t_{n-1}}^{t_n} f(t)\,dt,~~{\rm with}~~U^0=u_0,
\end{equation}
was developed in \cite{Mustapha2011}, where $\overline U(t)=(U^j+U^{j-1})/2$ for $t\in I_j$. Therein, the theoretical and numerical convergence results confirmed $O(\tau^{1+\alpha})$-rates in time over sufficiently time-graded meshes.    Both schemes \eqref{fully} and \eqref{GCN} are computationally similar, however the theoretical and numerical results show better convergence rates of the   $L1$ scheme.  

For computational purposes,    putting
\[\omega_{nj}=\int_{t_{j-1}}^{t_j}\omega_\alpha(t_n-s)\,ds~~~{\rm and}~~~\widehat \omega_{nj}=\int_{t_{j-1}}^{t_j}\int_s^{t_j}\omega_\alpha(t_n-q)\,dq\,ds,~~{\rm  for}~~j\le n.\]
 Hence \[
\int_{t_{n-1}}^{t_n}\partial_t^{1-\alpha} U(t)\,dt
	=(\fInt^\alpha U)(t_n)-(\fInt^\alpha U)(t_{n-1}),\]
with
\[
{ (\fInt^\alpha U)}(t_n)
	=\sum_{j=1}^n\int_{t_{j-1}}^{t_j}\omega_\alpha(t_n-s)\Big(U^{j-1}+(s-t_{j-1})\partial U^j\Big)\,ds
	=\sum_{j=1}^n(\omega_{nj}U^{j-1}+\widehat \omega_{nj}\partial U^j).
\]
 Then, the numerical scheme in \eqref{fully} is equivalent to 
\begin{multline} \label{fully 2}
 U^n+ \frac{\tau_n^\alpha}{\Gamma(\alpha+2)}\A U^n=U^{n-1}-\frac{\alpha\tau_n^\alpha}{\Gamma(\alpha+2)}\A U^{n-1}\\
-\sum_{j=1}^{n-1}\Big((\omega_{nj}-\omega_{n-1,j})\A U^{j-1}+(\widehat\omega_{nj}-\widehat\omega_{n-1,j})\A \partial  U^j\Big)+\int_{t_{n-1}}^{t_n} f(t)\,dt.
\end{multline}

\section{Error analysis}\label{sec: error analysis}
%%%%%%%%%%%%%%%%%%%%%%%%%%%%%%%%%%%%%%%%%%%%%%%%%%%%%%%%%%%%
In this section,  we study the error bounds from the time-stepping  scheme \eqref{fully}.   A preliminary estimate  will be derived in the next lemma.  For convenience, putting 
\begin{equation}\label{eq: eta(t)}
\eta(t)=\eta^n=\frac{1}{\tau_n}\int_{t_{n-1}}^{t_n}\Ba (u-\check u)(s)\,ds,~{\rm for}~ t\in I_n\,,\end{equation}
where the piecewise linear polynomial $\check u$ interpolates $u$ at the time nodes, that is, 
\[\check u(t)=u^{j-1}+(t-t_{j-1})\partial u^j\quad\text{for $t_{j-1}\le t\le t_j$ with $1\le j\le N$.}
\]

\begin{lemma}\label{lem: bound1 of theta} For $1\le n \le N,$ we have 
\[
 \|U^n-u(t_n)\|^2\le  
Ct_n^{1-\alpha}\sum_{j=1}^n \tau_j \|\eta^j\|_1^2\,.\]
\end{lemma}
\begin{proof} From equations  \eqref{eq: ibvp int} and \eqref{fully}, 
\[(U^j-u(t_j))-(U^{j-1}-u(t_{j-1}))+\int_{t_{j-1}}^{t_j} \A\Ba
(U-\check u)(t)\,dt
    =\tau_j \A\eta^j\,.
\]
Taking the inner product with $v:= \int_{t_{j-1}}^{t_j}\partial_t^{1-\alpha}(U-\check u)dt=
\int_{t_{j-1}}^{t_j}\fInt^\alpha (U-\check u)'\,dt$ (because  $U^0-\check u(0)=U^0-u_0=0$), and using 
\[A\Big(\int_{t_{j-1}}^{t_j}\fInt^\alpha (U-\check u)'\,dt,\int_{t_{j-1}}^{t_j}\fInt^\alpha(U-\check u)'\,dt\Big)\ge \beta \Big\|\int_{t_{j-1}}^{t_j}\fInt^\alpha(U-\check u)'\,dt\Big\|_1^2,\]
for some positive constant $\beta$ depends on $\Omega$ (due to the Poincar\'e inequality),    
\begin{multline*}
\tau_j\int_{t_{j-1}}^{t_j} \iprod{(U-\check u)',\fInt^\alpha(U-\check u)'}\,dt+
\beta\Big\|\int_{t_{j-1}}^{t_j}\fInt^\alpha(U-\check u)'\,dt\Big\|_1^2\\
{ \le} \tau_j\bigiprod{\A\eta^j,\int_{t_{j-1}}^{t_j}\fInt^\alpha(U-\check u)'\,dt}\,.\end{multline*}
An application of the Cauchy-Schwarz inequality leads to  
\[ 
\tau_j\bigiprod{\A\eta^j,\int_{t_{j-1}}^{t_j}\fInt^\alpha(U-\check u)'\,dt}
 { \le}
\frac{1}{2\beta}\tau_j^2\|\eta^j\|_1^2+
\frac{\beta}{2}\Big\|\int_{t_{j-1}}^{t_j}\fInt^\alpha(U-\check u)'\,dt\Big\|_1^2,\]
and consequently, 
\begin{equation} \label{eq: theta n0}
\int_{t_{j-1}}^{t_j}\iprod{(U-\check u)',\fInt^\alpha(U-\check u)'}\,dt{ \le}
C\tau_j\|\eta^j\|_1^2.\end{equation}
Summing over the variable $j$,
\[\int_0^{t_n}\iprod{(U-\check u)',\fInt^\alpha(U-\check u)'}\,dt{ \le}
C\sum_{j=1}^n \tau_j\|\eta^j\|_1^2.\]
Hence, using   Lemma \ref{lem: pointwise bound}, 
the desired bound is obtained.
\end{proof}

The current task is to    estimate  the candidate  $\|\eta^j\|_1$  which is very delicate. The approach is novel and some new ideas are used. Starting from the fact that  $(u-\check u)(0)=0$ (since $\check u$ interpolates $u$ at the time nodes $t_n$ for $0\le n\le N$), we observe 
\[\Ba (u-\check u)(t)=\fInt^\alpha (u-\check u)'(t)
=\sum_{i=1}^j \int_{t_{i-1}}^{\min\{t_i,t\}}\omega_\alpha(t-s)\Big(u'(s)-{ \partial u^i}\Big)\,ds,~~{\rm for}~t\in I_j.\]
However, %a simple application of Taylor series expansion with integral reminder yields  
\[
u(t_i)-u(t_{i-1})=\tau_i u'(t_i)-\frac{\tau_i^2}{2}u''(t_i)+\frac{1}{2}\int_{t_{i-1}}^{t_i}(q-t_{i-1})^2u'''(q)\,dq,\]
and hence, after some manipulations, one can show that
\begin{equation*}
\begin{aligned}
u'(s)-\partial u^i &=e_1(s)+e_2(s)\quad {\rm for}
~~s\in I_i,
\end{aligned}
\end{equation*}
where
\begin{align*}
e_1(s)
&=\int_s^{t_i}(q-s)u'''(q)\,dq-\frac{1}{2\tau_i}\int_{t_{i-1}}^{t_i}(q-t_{i-1})^2u'''(q)\,dq,\\
 e_2(s) &= \big(s-t_{i-1/2}\big)u''(t_i).
\end{align*}
Thus, splitting $\eta^j$ as: $\eta^j=\tau_j^{-1}\Big(\eta_1^j+\eta_2^j\Big)$, where 
 \begin{equation}\label{eq: eta21 eta22}
\eta_1^j=\int_{t_{j-1}}^{t_j} \fInt^\alpha e_1(t)\,dt\quad{\rm and}\quad
\eta_2^j=\int_{t_{j-1}}^{t_j}\fInt^\alpha e_2(t)\,dt.
\end{equation}
Estimating  $\|\eta_1^j\|_1$  will be the topic of the  next lemma. 
\begin{lemma}\label{lemma: bound of eta1}
For $1\le j\le N,$  we have
\[
\| \eta_1^j\|_1
    \le C\tau_j\tau^2 t_j^{3\alpha/2+\sigma-1-2/\gamma},\quad {\rm with}~~\gamma>2/(\sigma+\alpha/2).\]
\end{lemma}
\begin{proof}
Expanding  $\eta_1^j$ as
\[\eta_1^j=\sum_{i=1}^j\int_{t_{j-1}}^{t_j}\int_{t_{i-1}}^{\min\{t_i,t\}}\omega_\alpha(t-s)\,e_1 (s)\,ds\,dt.\]
 From  the definition of $e_1$ and the regularity assumption \eqref{eq: regularity}, for $s\in I_1,$
\begin{equation}\label{eq: bound of Ae1}
\begin{aligned}\|e_1(s)\|_1&\le \int_s^{t_1}(q-s)\|u'''(q)\|_1\,dq+\frac{1}{2t_1}\int_0^{t_1}q^2\|u'''(q)\|_1\,dq\\
&\le M\int_s^{t_1}q\,q^{\sigma+\alpha/2-3}\,dq+\frac{M}{2t_1}\int_0^{t_1}q^2q^{\sigma+\alpha/2-3}\,dq\\
&\le 
C\,\max\{s^{\sigma+\alpha/2-1},t_1^{\sigma+\alpha/2-1}\}\,.
\end{aligned}
\end{equation}
Hence, for $j=1,$   
\begin{equation}\label{eq: bound on I1}
\begin{aligned}
\|\eta_1^1&\|_1\le C\int_0^{t_1}\int_0^t\omega_\alpha(t-s)\max\{s^{\sigma+\alpha/2-1},t_1^{\sigma+\alpha/2-1}\}\,ds\,dt\\
&=  C\int_0^{t_1}\max\Big\{{ \frac{\Gamma(\sigma+\alpha/2)}{\Gamma(\sigma+3\alpha/2)}}t^{\sigma+3\alpha/2-1},{ \omega_{\alpha+1}(t)} t_1^{\sigma+\alpha/2-1}\Big\}\,dt\\
&= 
 C\max\Big\{{\frac{\Gamma(\sigma+\alpha/2)}{\Gamma(\sigma+3\alpha/2+1)}}t_1^{\sigma+3\alpha/2},
 { \omega_{\alpha+2}(t_1)} t_1^{\sigma+\alpha/2-1}\Big\}\le C\tau_1^{3\alpha/2+\sigma}\,.
 \end{aligned}
 \end{equation}
For the case  $j \ge 2,$ noting first that
\[
\|\eta_1^j\|_1\le \int_{t_{j-1}}^{t_j}\Big(\int_0^{t_1}
    \omega_\alpha(t-s)\,\|e_1 (s)\|_1\,ds
    +\sum_{i=2}^j\int_{t_{i-1}}^{\min\{t_i,t\}}
    \omega_\alpha(t-s)\,\|e_1 (s)\|_1\,ds\Big)dt.\]
For estimating  the first term on the right-hand side in the above inequality, using 
\[t-s\ge t_{j-1}-s=(j-1)^\gamma\Big(t_1-\frac{s}{(j-1)^\gamma}\Big)\ge  (j-1)^\gamma\Big(t_1-s\Big)\ge  (j/2)^\gamma\Big(t_1-s\Big),\]
and the achieved bound in \eqref{eq: bound of Ae1},
\begin{multline*}
\int_0^{t_1}
    \omega_\alpha(t-s)\,\|e_1 (s)\|_1\,ds\le Cj^{\gamma(\alpha-1)}\int_0^{t_1}(t_1-s)^{\alpha-1}\max\{s^{\sigma+\alpha/2-1},t_1^{\sigma+\alpha/2-1}\}\,ds\\
    = Ct_j^{\alpha-1}\tau^{\gamma(1-\alpha)} \int_0^{t_1}(t_1-s)^{\alpha-1}\max\{s^{\sigma+\alpha/2-1},t_1^{\sigma+\alpha/2-1}\}\,ds\\
    \le  C\,t_j^{\alpha-1}\tau^{\gamma(1-\alpha)}t_1^{3\alpha/2+\sigma-1}=  C\,\tau^2 t_j^{\alpha-1}t_1^{\alpha/2+\sigma-2/\gamma}\\
     \le C\,\tau^2 t_j^{3\alpha/2+\sigma-2/\gamma-1},\quad {\rm for}~~\gamma\ge 2/(\sigma+\alpha/2).\end{multline*}    
    On the other hand, for $s \in I_i$ with $i\ge 2,$  from  the definition of $e_1$, the regularity assumption \eqref{eq: regularity}, and the time mesh property \eqref{eq: mesh property 1},   we have 
\begin{multline*}
\|e_1(s)\|_1\le
    C\tau_i\int_{t_{i-1}}^{t_i}\|u'''(q)\|_1\,dq
    \le
    C\tau_i\int_{t_{i-1}}^{t_i}q^{\sigma+\alpha/2-3}\,dq\\
    \le  C\tau_i^2t_i^{\sigma+\alpha/2-3}\le C\tau^2t_i^{\sigma+\alpha/2-1-2/\gamma}\le C\tau^2s^{\sigma+\alpha/2-1-2/\gamma}\,.
\end{multline*}
 Thus,   \begin{align*}
\sum_{i=2}^j\int_{t_{i-1}}^{\min\{t_i,t\}}
    \omega_\alpha(t-s)\,\|e_1 (s)\|_1\,ds&\le C\tau^2\int_{t_1}^{t}(t-s)^{\alpha-1}s^{\sigma+\alpha/2-1-2/\gamma}\,ds\\
    &\le  C\tau^2\int_0^{t}(t-s)^{\alpha-1}s^{\sigma+\alpha/2-1-2/\gamma}\,ds\\
    &\le C\tau^2t^{3\alpha/2+\sigma-1-2/\gamma},\quad {\rm for}~~\gamma>2/(\sigma+\alpha/2).\end{align*}
 Gathering the above contribution and using again the time mesh property \eqref{eq: mesh property 1},
\begin{align*}
\| \eta_1^j\|_1
    \le C\tau_1^{3\alpha/2+\sigma}+&C\tau_j\tau^2 t_j^{3\alpha/2+\sigma-1-2/\gamma},\quad {\rm for}~~j\ge 2~~{\rm with}~~\gamma>2/(\sigma+\alpha/2).\end{align*}
From this bound, and the achieved estimate in \eqref{eq: bound on I1}, the desired result is obtained. 
\end{proof}

It remains to estimate $\|\eta_2^j\|_1$. The technical result in the next lemma is needed.
\begin{lemma}\label{lem: bound of Kij}
For $\gamma \ge 1,$ and  for a given positive sequence $\{a_i\}$, we have 
\[\sum_{i=2}^ja_i\Big|\K^{i-1,j-1}
    -\frac{\tau_{i-1}^3}{\tau_i^3}\K^{i,j}\Big|   \le C\tau\,  t_{j-1}^{\alpha-1/\gamma}\max_{i=2}^j \Big(a_i \tau_{i-1}^2\Big),\quad {\rm for}~~2\le i\le j,\]
    with \[
\K^{i,j}:=   
     \frac{1}{2}\int_{t_{i-1}}^{t_i}(s-t_{i-1})(t_i-s)
\omega_\alpha(t_j-s)\,ds,\quad{\rm for}~~~1\le i\le j\le N.\]
    \end{lemma}
\begin{proof}  For $\gamma =1,$ $\K^{i-1,j-1}
    -\frac{\tau_{i-1}^3}{\tau_i^3}\K^{i,j}=0$, and so, we have nothing to show. For $\gamma>1,$ from the  substitution $s=\tau_i^{-1}\Big((t_i-q)t_{i-2}+(q-t_{i-1})t_{i-1}\Big),$
we observe  \begin{equation}\K^{i-1,j-1}=\frac{1}{2}\frac{\tau_{i-1}^3}{\tau_i^3}\int_{t_{i-1}}^{t_i}
    (q-t_{i-1})(t_{i}-q)\omega_\alpha(t_{j-1}-s)\,dq.\label{eq:
    ki-1j-1}
\end{equation}
 Since  $t_{j-1}-s\le t_j-q$, $\K^{i-1,j-1}\ge \frac{\tau_{i-1}^3}{\tau_i^3}\K^{i,j}. $ This leads to  
 \begin{align*}&\sum_{i=2}^ja_i\Big|\K^{i-1,j-1}
    -\frac{\tau_{i-1}^3}{\tau_i^3}\K^{i,j}\Big|\\&= \frac{1}{2}\sum_{i=2}^ja_i\frac{\tau_{i-1}^3}{\tau_i^3}\int_{t_{i-1}}^{t_i}
   (q-t_{i-1})(t_{i}-q)[\omega_\alpha(t_{j-1}-s)-\omega_\alpha(t_j-q)]\,dq\\
   &\le \frac{1}{8}\max_{i=2}^j \Big(a_i \frac{\tau_{i-1}^3}{\tau_i}\Big) \sum_{i=2}^j\int_{t_{i-1}}^{t_i}[\omega_\alpha(t_{j-1}-s)-\omega_\alpha(t_j-q)]\,dq\\
   &\le\frac{1}{8}\max_{i=2}^j \Big(a_i \tau_{i-1}^2\Big)\sum_{i=2}^j\biggl(\frac{\tau_i}{\tau_{i-1}}\int_{t_{i-2}}^{t_{i-1}}
    \omega_\alpha(t_{j-1}-v)\,dv-\int_{t_{i-1}}^{t_i}\omega_\alpha(t_j-q)\,dq\biggr)\,.\end{align*}
   A simple manipulation shows that 
 \begin{multline*} \sum_{i=2}^j\int_{t_{i-1}}^{t_i}\omega_\alpha(t_j-q)\,dq=\int_{t_1}^{t_j}\omega_\alpha(t_j-q)\,dq\\=\omega_{\alpha+1}(t_j-t_1)
 \ge \omega_{\alpha+1}(t_{j-1})=\sum_{i=2}^j\int_{t_{i-2}}^{t_{i-1}}
    \omega_\alpha(t_{j-1}-v)\,dv, \end{multline*}
    and consequently, 
     \begin{multline*}\sum_{i=2}^ja_i\Big|\K^{i-1,j-1}
    -\frac{\tau_{i-1}^3}{\tau_i^3}\K^{i,j}\Big|\\
        \le \frac{1}{8}\max_{i=2}^j \Big(a_i \tau_{i-1}^2\Big)\sum_{i=2}^j\Big(\frac{\tau_i}{\tau_{i-1}}-1\Big)\int_{t_{i-2}}^{t_{i-1}}
    \omega_\alpha(t_{j-1}-q)\,dq.\end{multline*}
By the mesh properties in \eqref{eq: mesh property 1},  
  \[\frac{\tau_i}{\tau_{i-1}}-1= (\tau_i-\tau_{i-1})\tau_{i-1}^{-1}\le C\tau^2t_i^{1-2/\gamma}\tau^{-1} t_{i-1}^{1/\gamma-1}\le C\tau t_{i-1}^{-1/\gamma},~~{\rm for}~~i\ge 2,\] 
  and therefore, using $\gamma>1,$ 
  \begin{align*}
\sum_{i=2}^j\Big(\frac{\tau_i}{\tau_{i-1}}-1\Big)\int_{t_{i-2}}^{t_{i-1}}
    \omega_\alpha(t_{j-1}-q)\,dq &\le C\tau \sum_{i=2}^jt_{i-1}^{-1/\gamma}\int_{t_{i-2}}^{t_{i-1}}    \omega_\alpha(t_{j-1}-q)\,dq\\
&   \le C\tau  \int_0^{t_{j-1}}  q^{-1/\gamma}  \omega_\alpha(t_{j-1}-q)\,dq
   \le C\tau t_{j-1}^{\alpha-1/\gamma}\,.    
    \end{align*}
This completes the proof.
\end{proof}

Now, we are ready to bound $\|\eta_2^j\|_1$. 
\begin{lemma}\label{lemma: bound of eta2}
For $\eta_2^j$ defined as in~\eqref{eq: eta21 eta22} with $j\ge 1,$ we have
 \[\|\eta^j_2\|_1\le  C\tau^2 \tau_j  t_j^{3\alpha/2+\sigma-1-2/\gamma},\quad{\rm for}~~\gamma \ge 2/(\sigma+\alpha/2).\]
\end{lemma}
\begin{proof}
Splitting $\eta_2^j$ follows by reversing the order of integration then integrating by parts,
\begin{multline*}
\eta_2^j=\sum_{i=1}^ju''(t_i)
    \int_{t_{j-1}}^{t_j}\int_{t_{i-1}}^{\min\{t_i,t\}}(s-t_{i-1/2})\omega_\alpha(t-s)
    \,ds\,dt\\
    =\sum_{i=1}^{j-1}u''(t_i)\int_{t_{i-1}}^{t_i}(s-t_{i-1/2})[\omega_{\alpha+1}(t_j-s)-\omega_{\alpha+1}(t_{j-1}-s)]
    \,ds\\ 
    +u''(t_j)\int_{t_{j-1}}^{t_j}(s-t_{j-1/2})\omega_{\alpha+1}(t_j-s)
    \,ds\\
    =\sum_{i=1}^{j-1}u''(t_i)[\K^{i,j-1}-\K^{i,j}] -u''(t_j)\K^{j,j}\,.
\end{multline*}
Thus, $\eta_2^j$ can be decomposed as: $\eta_2^j=-\eta_{2,1}^j-\eta_{2,2}^j-\eta_{2,3}^j$, with 
\begin{gather}
\eta_{2,1}^j= \sum_{i=2}^j[u''(t_i)-u''(t_{i-1})]\K^{i,j},
    \label{eq: B1j}\\
%\eta_{2,2}^j= u''(t_{j-1})[\K^{j,j}_\alpha-\K^{j-1,j-1}_\alpha],\\
\eta_{2,2}^j=\sum_{i=2}^ju''(t_{i-1})\biggl(1-\frac{\tau_{i-1}^3}{\tau_i^3}\biggr)
\K^{i,j},\label{eq: B2j}\\
  \eta^j_{2,3}=u''(t_1)\K^{1,j}-\sum_{i=2}^ju''(t_{i-1})\biggl(\K^{i-1,j-1}
    -\frac{\tau_{i-1}^3}{\tau_i^3}\K^{i,j}\biggr).\label{eq: B3j}
\end{gather}  By    the regularity assumption in \eqref{eq: regularity}  and the time mesh properties in \eqref{eq: mesh property 1},   
    \begin{equation}
    \begin{aligned}
  \|\eta_{2,1}^j\|_1 &\le C\sum_{i=2}^j \tau_i^2\int_{t_{i-1}}^{t_i}
   \|u'''(q)\|_1\,dq\int_{t_{i-1}}^{t_i}(t_j-s)^{\alpha-1}\,ds\\
   & \le C\sum_{i=2}^j \tau_i^3t_i^{\sigma+\alpha/2-3}
   \int_{t_{i-1}}^{t_i}(t_j-s)^{\alpha-1}\,ds \\
   &\le C\tau^3\sum_{i=2}^j t_i^{\sigma+\alpha/2-3/\gamma}
   \int_{t_{i-1}}^{t_i}(t_j-s)^{\alpha-1}\,ds
 \\ &\le C\tau^3 
   \int_{t_1}^{t_j}s^{\sigma+\alpha/2-3/\gamma}(t_j-s)^{\alpha-1}\,ds
    \le C\tau^3t_j^{\sigma+3\alpha/2-3/\gamma},
    \end{aligned}\label{eq: star1}
   \end{equation}
for $\gamma \ge 2/(\sigma+\alpha/2).$ The next task is to estimate $\eta^j_{2,2}.$ Seeing that 
       $$1-{\tau_{i-1}^3}/{\tau_i^3}\le 3
    \tau_i^{-1}(\tau_i-\tau_{i-1})\le C\tau^2\tau_i^{-1}t_i^{1-2/\gamma}$$ (the third time mesh property in \eqref{eq: mesh property 1} is used here), yields 
\begin{equation}\label{eq: B21}
\begin{aligned}
\|\eta^j_{2,2}\|_1&\le  C\tau^2\sum_{i=2}^j\tau_i^{-1}t_i^{1-2/\gamma} \|u''(t_{i-1})\|_1 \tau_i^2\int_{t_{i-1}}^{t_i}(t_j-s)^{\alpha-1}\,ds\\
&\le C\tau^3\sum_{i=2}^j \int_{t_{i-1}}^{t_i}(t_j-s)^{\alpha-1}s^{\sigma+\alpha/2-3/\gamma}\,ds\\
&\le C\tau^3\int_{t_1}^{t_j}(t_j-s)^{\alpha-1}s^{\sigma+\alpha/2-3/\gamma}\,ds\\
&\le C\tau^3t_j^{\sigma+3\alpha/2-3/\gamma},~~{\rm for}~~\gamma \ge 2/(\sigma+\alpha/2).
\end{aligned}
\end{equation}
To estimate $\eta^j_{2,3}$, we use again the regularity assumption in \eqref{eq: regularity} and then apply Lemma \ref{lem: bound of Kij}  with $t_{i-1}^{\sigma-2}$ in place of $a_i$, and get  
        \begin{align*}\sum_{i=2}^j\|u''(t_{i-1})\|_1\Big|\K^{i-1,j-1}
    -\frac{\tau_{i-1}^3}{\tau_i^3}\K^{i,j}\Big|
     &\le C\sum_{i=2}^j t_{i-1}^{\sigma-2}\Big|\K^{i-1,j-1}
    -\frac{\tau_{i-1}^3}{\tau_i^3}\K^{i,j}\Big|\\
     &\le C\tau  t_{j-1}^{\alpha-1/\gamma}\max_{i=2}^j \Big(t_{i-1}^{\sigma+\alpha/2-2} \tau_{i-1}^2\Big)\\
    &\le C\tau^3  t_{j-1}^{\alpha-1/\gamma}\max_{i=2}^j \Big(t_{i-1}^{\sigma+\alpha/2-2/\gamma}\Big)\\
   &\le C\tau^3  t_{j-1}^{3\alpha/2+\sigma-3/\gamma},~~{\rm for}~~\gamma \ge 2/(\sigma+\alpha/2).\end{align*}
       By   the definition of  $\eta^j_{2,3}$ and the above contribution, we  obtain
      \begin{equation}\label{eq: bound of eta1,23}
\begin{aligned}\|\eta^j_{2,3}\|_1&\le \|u''(t_1)\|_1\K^{1,j}+C\tau^3  t_{j-1}^{3\alpha/2+\sigma-3/\gamma}\\
    &\le C\tau_1^{\sigma+\alpha/2}\int_0^{t_1}\omega_\alpha(t_j-s)\,ds  + C\tau^3  t_{j-1}^{3\alpha/2+\sigma-3/\gamma}\\
       &\le C\tau_1^{3\alpha/2+\sigma}+  C\tau^3  t_{j-1}^{3\alpha/2+\sigma-3/\gamma},\quad{\rm for}~~\gamma\ge 2/(\sigma+\alpha/2).\end{aligned}
       \end{equation}
     Therefore, to complete the proof, combining  the estimates from \eqref{eq: star1},
\eqref{eq: B21} and  \eqref{eq: bound of eta1,23}, and use the inequality $\tau\le C \tau_j t_j^{1/\gamma-1}$ (follows from the second mesh property  in \eqref{eq: mesh property 1}). \end{proof}

We are ready now to estimate the pointwise error.  The proof relies on the achieved results in Lemmas \ref{lem: bound1 of theta}, \ref{lemma: bound of eta1} and  \ref{lemma: bound of eta2}. As mentioned earlier, the numerical results encapsulate that the imposed assumption on $\gamma$ is not sharp.
% especially as  $\alpha$ approaches $0$ which might not an interesting case from the practical point of slow diffusion.       
\begin{theorem}\label{thm: time convergence} Let $U$ be the time-stepping solution defined by \eqref{fully} and let $u$ be the solution of the fractional reaction-diffusion problem \eqref{eq: ibvp}. Assume that $u$  satisfies the regularity assumptions in \eqref{eq: regularity}. If  the time mesh exponent  $\gamma$ is  greater than the maximum of $\{2/(\sigma+\alpha/2),2/(\sigma+3\alpha/2-1/2)\}$ with $\sigma+3\alpha/2-1/2>0,$  then   
\[\|U^n-u(t_n)\| \le  
C\tau^2 t_n^{\sigma+\alpha-2/\gamma}\le C\tau^2,\quad{\rm for}~~1\le n \le N\,.\]
\end{theorem}
\begin{proof}  From  the  decomposition $\eta^j=\tau_j^{-1}\big(\eta_1^j+\eta_2^j\big)$, and the established bounds of $\eta_1^j$ and $\eta_2^j$ in Lemma 
 \ref{lemma: bound of eta1} and Lemma \ref{lemma: bound of eta2}, respectively, 
\begin{equation}\label{eq: sum theta}
\begin{aligned}  
\sum_{j=1}^n \tau_j\|\eta^j\|_1^2
&\le  C\tau^4 \sum_{j=1}^n \tau_j  t_j^{2(3\alpha/2+\sigma-1-2/\gamma)}\\
&\le C\tau^4 \int_{t_1}^{t_n}  t^{2(3\alpha/2+\sigma-1-2/\gamma)}\,dt\\
&\le C\tau^4 \max\{t_1^{2\sigma+3\alpha-4/\gamma-1},t_n^{2\sigma+3\alpha-4/\gamma-1}\},\quad{\rm for}~~\gamma>2/(\sigma+\alpha/2) \,.\end{aligned}
\end{equation}
Inserting this in the achieved bound in Lemma \ref{lem: bound1 of theta} and using $\gamma\ge 2/(\sigma+3\alpha/2-1/2)$ will  complete the proof. 
\end{proof}
\section{{Fully}-discrete solution} \label{sec: fully-discrete}
To compute our numerical solution, we therefore seek a fully-discrete solution $U_h$ by discretizing \eqref{fully}  in space via the standard Galerkin finite element method. To this end,   let  $\mathcal{T}_h$ be a family of regular (conforming) triangulation  of the  domain $\overline{\Omega}$ and let $h=\max_{K\in \mathcal{T}_h}(\mbox{diam}K),$ where $h_{K}$ denotes the diameter  of the element  $K.$  Let $V_h \subset H^1_0(\Omega)$ denote the usual space of continuous, piecewise-linear functions on  $\mathcal{T}_h$ that vanish on $\partial \Omega$. Let $\W(V_h)\subset  C([0,T];V_h)$ denote the space of linear polynomials on $[t_{n-1},t_n]$ for $1\le  n\le N$,  with coefficients in $V_h.$

Taking the inner of \eqref{fully} with a  test function $\chi \in H^1_0(\Omega)$, and apply the first Green identity. Then, the semi-discrete $L1$   solution $U$  satisfies
\begin{equation} \label{fully0}
 \iprod{U^n-U^{n-1},\chi}+ \int_{t_{n-1}}^{t_n}A(\partial_t^{1-\alpha}U(t),\chi)\,dt= \int_{t_{n-1}}^{t_n} \iprod{f(t),\chi}\,dt,
 ~~{\rm with}~~U^0=u_0,
 \end{equation}
Motivated by this, our fully-discrete computational solution  $U_h\in \W(V_h)$  is defined as:  for $1\le n\le N,$
\begin{equation} \label{fully 2 2}
\iprod{U_h^n-U_h^{n-1},v_h}+ \int_{t_{n-1}}^{t_n}A(\partial_t^{1-\alpha}U_h(t),v_h)\,dt= { \int_{t_{n-1}}^{t_n}\iprod{f(t),v_h}\,dt}\quad \forall~ v_h\in V_h,
\end{equation}
 with $U_h^0=R_h u_0$, where $R_h:H^1_0(\Omega) \to V_h$ is the Ritz projection defined by \[
 A(R_h w,v_h)=A(w,v_h),\quad \forall~v_h \in V_h.\]

Following the derivation used to obtain  \eqref{fully 2}, the scheme in \eqref{fully 2 2} is equivalent to 
\[\begin{aligned}
&\iprod{U_h^n,v_h}+ \frac{\tau_n^\alpha}{\Gamma(\alpha+2)}A(U_h^n,v_h)=\iprod{U_h^{n-1},v_h}-\frac{\alpha\tau_n^\alpha}{\Gamma(\alpha+2)}A(U_h^{n-1},v_h)\\
&- A\bigg(\sum_{j=1}^{n-1}\Big((\omega_{nj}-\omega_{n-1,j})U_h^{j-1}+(\widehat\omega_{nj}-\widehat\omega_{n-1,j})\partial U_h^j\Big),v_h\bigg) +\int_{t_{n-1}}^{t_n} \iprod{f(t),v_h}\,dt\,.
\end{aligned}\] 
 For $1\le p\le d_h:=\dim V_h$, let $\phi_p\in V_h$ denote the $p$th
basis function associated with the $p$th interior node $\vec{x}_p$, so that $\phi_p(\vec{x}_q)=\delta_{pq}$~and
\[U^n_h(\vec{x})=\sum_{p=1}^{d_h}u^n_h(\vec x_p)\phi_p(\vec{x}).\]
We define $d_h\times d_h$ matrices: $\matM=[\iprod{\phi_q,\phi_p}],$ $\matG=[A(\phi_{q},\phi_p)],$ and the
$d_h$-dimensional column vectors ${\bf U}_h^n$  and ${\bf F}^n$ with components~$U^n_h(\vec x_p)$ and $\int_{t_{n-1}}^{t_n} \iprod{f(t),\phi_p}\,dt,$ respectively. Therefore, the fully-discrete scheme \eqref{fully 2 2} has the following matrix representations: 
% at the $n$th time step we  solve 
\begin{multline*}
\Big(\matM+ \frac{\tau_n^\alpha}{\Gamma(\alpha+2)}\matG\Big){\bf U}_h^n
	=\Big(\matM-\frac{\alpha\tau_n^\alpha}{\Gamma(\alpha+2)}\matG\Big){\bf U}_h^{n-1}\\
	-\sum_{j=1}^{n-1}\Big(\bigl(\omega_{nj}-\omega_{n-1,j}\bigr)\matG
	{\bf U}_h^{j-1}+(\widehat\omega_{nj}-\widehat\omega_{n-1,j})\matG
	\partial {\bf U}_h^j\Big) +{\bf F}^n.
\end{multline*}

Therefore, at each time level $t_n,$  the numerical scheme \eqref{fully 2 2} reduces to a finite square linear system, and so the existences of $U_h^n$ follows from its uniqueness. The latter follows from the fact that both  matrices $\matM$ and $\matG$ are positive definite.

Turning now into the error analysis, we introduce the following notations:
\[{ \theta(t)=U_h(t)-R_h \check u(t)}~~{\rm and}~~  \rho(t)=u(t)-R_hu(t), \quad{\rm and}\,.\]
Since  $\check u$ interpolates $u$ at the time nodes, $\theta^n=U_h^n-{  R_h \check u(t_n)}=U_h^n- { R_h u(t_n)}.$  Thence,  the pointwise time error $U_h^n-u(t_n)$ can be decomposed as 
\begin{equation}\label{eq: decomposition}
U_h^n-u(t_n)=[U_h^n-R_h u(t_n)]-[u(t_n)-R_h u(t_n)]=\theta^n-\rho^n.\end{equation}

 The estimate of the second term follows easily from the Ritz projector approximation property and the first regularity assumption in \eqref{eq: regularity},
\begin{equation}\label{eq: v-Rhv}
\|\rho(t_n)\|\le Ch^2\|u(t_n)\|_{2}\le Ch^2,
    \quad\text{for $0\le n\le N$.}
\end{equation}
The next duty is to estimate  $\theta^n$. { From  the weak formulation of problem \eqref{eq: ibvp};
\[\iprod{u(t_j)-u(t_{j-1}),\chi}+ \int_{t_{j-1}}^{t_j}A(\partial_t^{1-\alpha}u(t),\chi)\,dt= \int_{t_{j-1}}^{t_j}\iprod{f(t),\chi}\,dt\quad \forall~ \chi\in H^1_0(\Omega),\]
 the numerical scheme \eqref{fully 2 2}, and the decomposition in \eqref{eq: decomposition}, we have 
 \begin{multline*}
\tau_j\iprod{\partial\theta^j,v_h}+\int_{t_{j-1}}^{t_j} A\bigl(\Ba
(U_h-\check u)(t),v_h\bigr)\,dt
    \\
    =\tau_j\iprod{\partial \rho^n,v_h}+\int_{t_{j-1}}^{t_j} A\bigl(\Ba
(u-\check u)(t),v_h\bigr)\,dt,\quad \forall~ v_h\in V_h\,.
\end{multline*}
From the orthogonality property  of the Ritz projection, and the definition of $\eta$ in \eqref{eq: eta(t)}, 
\begin{equation}\label{eq: new Eq}
\tau_j\iprod{\partial\theta^j,v_h}+\int_{t_{j-1}}^{t_j} A\bigl(\Ba
\theta(t),v_h\bigr)\,dt
    =\tau_j[\iprod{\partial \rho^j,v_h}+ A(\eta^j,v_h)],~ \forall~ v_h\in V_h\,.
\end{equation}
Since  $\theta^0=U_h^0-R_h\check u(0)=R_h u_0-R_h u_0=0$, $\int_{t_{j-1}}^{t_j}\partial_t^{1-\alpha}\theta(t)dt=
\int_{t_{j-1}}^{t_j}\fInt^\alpha \theta'(t)\,dt.$ Now, setting  $v_h= \int_{t_{j-1}}^{t_j}\fInt^\alpha \theta'(t)\,dt$ and applying the Poincar\'e inequality, then  
the second term in \eqref{eq: new Eq} is  $\ge \beta \|\int_{t_{j-1}}^{t_j}\fInt^\alpha\theta'(t)\,dt\|_1^2,$
for some positive constant $\beta$ depends on $\Omega.$ This and the fact that $\partial\theta^j=\theta'(t)$ (constant in time) for $t\in I_j$, lead to   
 \begin{multline*}
\tau_j\int_{t_{j-1}}^{t_j} \iprod{\theta',\fInt^\alpha\theta'}\,dt+
\beta\Big\|\int_{t_{j-1}}^{t_j}\fInt^\alpha\theta'\,dt\Big\|_1^2\\
\le \tau_j\bigiprod{\partial \rho^j,\int_{t_{j-1}}^{t_j}\fInt^\alpha\theta'\,dt}+ \tau_jA\Big(\eta^j,\int_{t_{j-1}}^{t_j}\fInt^\alpha\theta'\,dt\Big)\,.\end{multline*}
By the Cauchy-Schwarz inequality, the last term is
\[ 
 \le 
\frac{1}{2\beta}\tau_j^2\|\eta^j\|_1^2+
\frac{\beta}{2}\Big\|\int_{t_{j-1}}^{t_j}\fInt^\alpha\theta'\,dt\Big\|_1^2,\]
and consequently, 
\begin{equation} \label{eq: theta n01}
\tau_j\int_{t_{j-1}}^{t_j}\iprod{\theta',\fInt^\alpha\theta'}\,dt+\frac{\beta}{2}\Big\|\int_{t_{j-1}}^{t_j}\fInt^\alpha\theta'\,dt\Big\|_1^2\le 
\tau_j\int_{t_{j-1}}^{t_j}\iprod{{\check \rho}',\fInt^\alpha\theta'}\,dt+C\tau_j^2\|\eta^j\|_1^2,
\end{equation}
where $\check \rho(t)=\rho^{j-1}+(t-t_{j-1})\partial \rho^j$ for $t\in I_j$. Dividing both sides by $\tau_j$, and then, summing over the variable $j$ and using the inequality      
\begin{equation}\label{eq: alpha dep}
\int_0^{t_n}\iprod{{\check \rho}',\fInt^\alpha \theta'}\,dt\le 
\frac{1}{2(1-\alpha)^2}\int_0^{t_n}\iprod{{\check \rho}',\fInt^\alpha {\check \rho}'}\,dt+\frac{1}{2}\int_0^{t_n}\iprod{\theta',\fInt^\alpha\theta'}\,dt,\end{equation}
(Lemma \ref{lem: alpha dep} is used here), we reach
\[\int_0^{t_n}\iprod{\theta',\fInt^\alpha\theta'}\,dt\le 
\frac{1}{(1-\alpha)^2}\int_0^{t_n}\iprod{{\check \rho}',\fInt^\alpha {\check \rho}'}\,dt+C\tau_j\|\eta^j\|_1^2\,.\]
Thanks to Lemma \ref{lem: pointwise bound},
\begin{equation}\label{eq: theta discrete}
 \|\theta^n\|^2\le  
C t_n^{1-\alpha}\Big(\int_0^{t_n}|\iprod{{\check \rho}',\fInt^\alpha {\check \rho}'}|\,dt+ \sum_{j=1}^n \tau_j\|\eta^j\|_1^2\Big)\,.\end{equation}}
To estimate  $\int_0^{t_n}|\iprod{{\check \rho}',\fInt^\alpha {\check \rho}'}|\,dt$,  split it as 
 (recall that ${\check \rho}'(t)=\partial \rho^j$ on $I_j$) 
\begin{multline}\label{eq: splitting of xi}
\int_0^{t_n}|\iprod{{\check \rho}'(t),\fInt^\alpha{\check \rho}'(t)}|\,dt\le 
\|\partial\rho^1\|^2 \int_0^{t_1}\int_0^{t}\omega_\alpha(t-s)dsdt\\
+\sum_{j=2}^n\|\partial\rho^j\|\,\|\partial\rho^1\|\int_{t_{j-1}}^{t_j}\int_0^{t_1}\omega_\alpha(t-s)dsdt\\
+\sum_{j=2}^n\|\partial\rho^j\|\sum_{i=2}^j \|\partial\rho^i\| \int_{t_{j-1}}^{t_j}\int_{t_{i-1}}^{\min\{t_i,t\}}\omega_\alpha(t-s)dsdt\,.
\end{multline}
By the definition of the function $\rho$,  the Ritz projection error  bound in (\ref{eq: v-Rhv}) with $u'$ in place of $u$, and the regularity assumption \eqref{eq: regularity}, we obtain 
\begin{multline}\label{eq: bound of rho} 
\|\partial\rho^j\|=\tau_j^{-1}\Big\|\int_{t_{j-1}}^{t_j}(R_h u'-u')(s)\,ds\Big\|\\
\le
Ch^2\tau_j^{-1}\int_{t_{j-1}}^{t_j}\|u'(s)\|_2\,ds\le
Ch^2\tau_j^{-1}\int_{t_{j-1}}^{t_j}s^{\sigma-1}\,ds,\quad{\rm for}~~j\ge 1.\end{multline}
{ If $\sigma-1\ge 0$ (which might not be practically the case), then $\|\partial\rho^j\|\le Ch^2 t_j^{\sigma-1}.$ Thus, 
\begin{equation}\label{eq: bound of phi Ialpha phi n}\int_0^{t_n}|\iprod{{\check \rho}',\fInt^\alpha{\check \rho}'}|\,dt\le Ch^4 t_n^{2\sigma-2}\int_0^{t_n}\int_0^t\omega_\alpha(t-s)dsdt\le Ch^4 t_n^{2\sigma+\alpha-1},~~ {\rm for}~\sigma \ge 1\,.
\end{equation}
Now, turning into the case $\sigma-1<0$, which is probably more interesting. Assuming that $\sigma>(1-\alpha)/2$, a similar bound will be achieved next,  see \eqref{eq: bound of phi Ialpha phi}}. Using \eqref{eq: bound of rho}, the first term on the right-hand side of  \eqref{eq: splitting of xi} is bounded by
\[
Ch^4\tau_1^{ -2}\Big(\int_0^{t_1}s^{\sigma-1}\,ds\Big)^2\omega_{\alpha+2}(t_1)\le Ch^4 t_1^{2\sigma+\alpha-1}\,.\]
Using \eqref{eq: bound of rho} and the first time mesh property in \eqref{eq: mesh property 1},  the second candidate on the right-hand side of  \eqref{eq: splitting of xi} is 
\begin{align*}&\le 
Ch^4\sum_{j=2}^n  
t_j^{\sigma-1}\frac{1}{\tau_1}\int_0^{t_1}s^{\sigma-1}\,ds \int_{t_{j-1}}^{t_j}\int_0^{t_1}\omega_\alpha(t-s)\,ds\,dt\\
&\le 
Ch^4\sum_{j=2}^n  
(t_1t_j)^{\sigma-1}\int_{t_{j-1}}^{t_j}\int_0^{t_1}\omega_\alpha(t-s)\,ds\,dt\\
&\le 
Ch^4\sum_{j=2}^n  
\int_{t_{j-1}}^{t_j}t^{\sigma-1}\int_0^{t_1}s^{\sigma-1}\omega_\alpha(t-s)\,ds\,dt\\
&\le 
Ch^4 
\int_{t_1}^{t_n} t^{\sigma-1}\int_0^ts^{\sigma-1}\omega_\alpha(t-s)\,ds\,dt
\le  Ch^4\int_{t_1}^{t_n} t^{2\sigma+\alpha-2}\,dt,
\end{align*}
while,  the last  term in \eqref{eq: splitting of xi} is 
\begin{align*}&\le 
Ch^4\sum_{j=2}^n  
t_j^{\sigma-1}\sum_{i=2}^j\frac{1}{\tau_i}\int_{I_i}s^{\sigma-1}\,ds \int_{t_{j-1}}^{t_j}\int_{t_{i-1}}^{\min\{t_i,t\}}\omega_\alpha(t-s)\,ds\,dt\\
&\le 
Ch^4\sum_{j=2}^n  
t_j^{\sigma-1}\sum_{i=2}^j t_i^{\sigma-1} \int_{t_{j-1}}^{t_j}\int_{t_{i-1}}^{\min\{t_i,t\}}\omega_\alpha(t-s)\,ds\,dt\\
&\le 
Ch^4\sum_{j=2}^n\sum_{i=1}^j  \int_{t_{j-1}}^{t_j}t^{\sigma-1}\int_{t_{i-1}}^{\min\{t_i,t\}}s^{\sigma-1}\omega_\alpha(t-s)\,ds\,dt\\
&\le 
Ch^4\int_{t_1}^{t_n} t^{\sigma-1}\int_0^t s^{\sigma-1}\omega_\alpha(t-s)\,ds\,dt
\le  Ch^4\int_{t_1}^{t_n} t^{2\sigma+\alpha-2}\,dt\,.
\end{align*}
 Therefore, gathering the above estimates, we conclude that
\begin{equation}\label{eq: bound of phi Ialpha phi}
\int_0^{t_n}|\iprod{{\check \rho}',\fInt^\alpha{\check \rho}'}|\,dt \le Ch^4t_n^{2\sigma+\alpha-1},\quad {\rm for}~~(1-\alpha)/2<\sigma<1\,.
\end{equation}

From the decomposition \eqref{eq: decomposition}, the Ritz projection in \eqref{eq: v-Rhv}, the inequality in \eqref{eq: theta discrete}, the achieved bounds in \eqref{eq: sum theta} and \eqref{eq: bound of phi Ialpha phi n}, and the above estimate,  the error result  in the next convergence theorem holds true. It is  claimed that for a sufficiently time-graded mesh, the proposed  fully-discrete scheme is second-order accurate in both time and space. The numerical results  in the forthcoming section confirm that the imposed assumption on the time mesh exponent $\gamma$ is pessimistic. Furthermore, these results also illustrate $O(h^2)$-rates of convergence in space, although  the imposed condition $\sigma>(1-\alpha)/2$ in the next theorem is not satisfied. Indeed, for the   semi-discrete Galerkin method in space for problem \eqref{eq: ibvp}, an  $O(h^2)$-rate of convergence was carried out without this assumption \cite{KaraaMustaphaPani2018}.
\begin{theorem}\label{thm: CR} Let $U_h$ be the numerical solution defined by \eqref{fully 2 2} and let $u$ be the solution of the fractional reaction-diffusion problem \eqref{eq: ibvp}. Assume that $u$  satisfies the regularity assumptions in \eqref{eq: regularity} with $\sigma>(1-\alpha)/2$. If the time mesh exponent  $\gamma >\max \{2/(\sigma+\alpha/2),2/(\sigma+3\alpha/2-1/2)\},$ then   
\[\|U_h^n-u(t_n)\| \le  
{ C(\tau^2+h^2)},\quad{\rm for}~~1\le n \le N.\]
\end{theorem}
We end this section with the following remark. 
\begin{remark}\label{remark: blows up} Due to the use of the inequality in  \eqref{eq: alpha dep}, the coefficient $C$ in \eqref{eq: theta discrete} blows up as $\alpha \to 1^-$.  { To control  this phenomena, Lemma \ref{lem: alpha dep} (and consequently, the inequality in \eqref{eq: alpha dep}) should be avoided. Since  ${\check \rho}'(t)=\partial \rho^j$ for $t\in I_j$, an application of  the Cauchy-Schwarz inequality yields 
\[\tau_j\int_{t_{j-1}}^{t_j}\iprod{{\check\rho}',\fInt^\alpha \theta'}\,dt=
\bigiprod{\tau_j \partial \rho^j, \int_{t_{j-1}}^{t_j}\fInt^\alpha \theta'\,dt}\le 
C\tau_j^2\|\partial\rho^j\|^2+\frac{\beta}{2}\bigg\|\int_{t_{j-1}}^{t_j}\fInt^\alpha \theta'(t)\,dt\bigg\|_1^2.\]
Substitute this in \eqref{eq: theta n01} gives 
\[
\int_{t_{j-1}}^{t_j}\iprod{\theta',\fInt^\alpha\theta'}\,dt\le 
C\tau_j\|\partial\rho^j\|^2+C\tau_j\|\eta^j\|_1^2\,.\]}
Summing over $j$, follows by using Lemma \ref{lem: pointwise bound} with $\theta$ in place of $\phi$,  we notice that  
\[
 \|\theta^n\|^2\le  
Ct_n^{1-\alpha} \Big(\sum_{j=1}^n \tau_j\|\partial \rho^j\|^2 + \sum_{j=1}^n \tau_j\|\eta^j\|_1^2\Big)\,,\]
where the constant $C$ in the above  bound does not blowup as $\alpha \to 1^-.$ 

The remaining exercise is to estimate $\sum_{j=1}^n \tau_j^{-1}\|\partial \rho^j\|^2$. From \eqref{eq: bound of rho}, 
\begin{multline*}
\sum_{j=1}^n \tau_j\|\partial \rho^j\|^2\le Ch^4\sum_{j=1}^n \tau_j^{-1}\Big(\int_{t_{j-1}}^{t_j}s^{\sigma-1}\,ds\Big)^2
\le Ch^4\sum_{j=1}^n \int_{t_{j-1}}^{t_j}s^{2\sigma-2}\,ds\\
= Ch^4\int_0^{t_n}s^{2\sigma-2}\,ds\le Ch^4 t_n^{2\sigma-1},~~{\rm for}~~\sigma>1/2.\end{multline*}
For $\sigma =1/2,$ these steps can be slightly adjusted to show  an $O(h^4)$  bound of the above term, but with a logarithmic coefficient ${\rm log}(t_n/t_1)$ for $n\ge 2.$ 
\end{remark}

%%%%%%%%%%%%%%%%%%%%%%%%%%%%%%%%%%%%%%%%%%%%%%%%%%%%%%%%%%5
\section{Numerical  results}\label{sec: numerical results}
%%%%%%%%%%%%%%%%%%%%%%%%%%%%%%%%%%%%%%%%%%%%%%%%%%%%%%%%%%%%
To support the achieved theoretical  convergence results in Theorems \ref{thm: time convergence} and \ref{thm: CR}, this section is devoted to perform some numerical experiments (on a typical test problem).  In the fractional model problem \eqref{eq: ibvp}, we choose  $u_0(x,y)=x(1-x)$, $\kappa_\alpha=d = 1$, and  $f=0.$   The time and space domains are chosen to be the intervals $[0,1]$ and  $(0,1),$ respectively.   Separation of variables yields the series  representation of the solution:
\begin{equation}\label{eq: u series}
u(x,t)=8\sum_{m=0}^\infty \lambda_m^{-3}\sin(\lambda_m x) E_{\alpha}(-\lambda_m^2 t^{\alpha}),\quad \lambda_m:=(2m+1)\pi-1,
\end{equation}
where $E_{\alpha}(t):=\sum_{p=0}^\infty\frac{t^p}{\Gamma(\alpha p+1)}$ is the Mittag-Leffler function. 
  
The initial data  $u_0 \in \dot H^{2.5^-}(\Omega)\cap H^1_0(\Omega)$. Thus, as expected from  the regularity analysis in \cite{McLean2010,McLeanMustaphaAliKnio2019}, the regularity properties in \eqref{eq: regularity} hold true for $\sigma=\alpha^-/4.$ 

For the numerical illustration of the convergence rates from the  time-stepping $L1$ scheme, we refine the spatial (uniform) mesh size $h$  so that the time errors are dominant. Therefore, by Theorems \ref{thm: time convergence} and \ref{thm: CR}, we expect to observe $O(\tau^2)$-rates of convergence for  $\gamma>\max\{2/(\sigma+\alpha/2),2/(\sigma+3\alpha/2-1/2)\}=\max\{8/(3\alpha^-),8/(7\alpha^- -2)\}$, with $\sigma+3\alpha/2-1/2>0$.  However, the results in Tables \ref{table 1}--\ref{table 3} are more optimistic, $O(\tau^2)$-rates  were  observed for  $\gamma \ge 2/(\sigma+\alpha) = 8/(5\alpha^-)$, for different values of $\alpha.$  Moreover, these results confirm that the condition $\sigma+3\alpha/2-1/2>0$ is not necessary.

In all tables and figures, we evaluated the series solution $u$ in \eqref{eq: u series} of problem \eqref{eq: ibvp} by truncating the Fourier series in \eqref{eq: u series} after $60$ terms. To measure the error in the numerical solution, we computed 
\[E_{N,M}:=\max_{1\le n\le N} \|U_h^n-u(t_n)\|,\]
 where $N$ is the number of time subintervals, while $M$ is the number of uniform space mesh elements.  Noting that, the spatial $L_2$-norm was evaluated using the two-point Gauss quadrature rule on the finest spatial mesh.  The convergence rates $r_t$ (in time) and $r_x$ (in space) were calculated from the relations 
 \begin{align*}
  r_t&\approx {\rm log}_2\Big(E_{N,M}/E_{2N,M}\Big),\quad{\rm when}~~~h^{r_x} \ll \tau^{r_t},\\
 r_x&\approx {\rm log}_2\Big(E_{N,M}/E_{N,2M}\Big),\quad{\rm when}~~~\tau^{r_t} \ll h^{r_x}.\end{align*}
\begin{table}
\begin{center}
\begin{tabular}{|c|cc|cc|cc|cc|}
\hline
$M$&\multicolumn{2}{c|}{$\gamma=1$ }
&\multicolumn{2}{c|}{$\gamma=2$}
&\multicolumn{2}{c|}{$\gamma=3$}&\multicolumn{2}{c|}{$\gamma=4$}\\
\hline
   20& 3.40e-02&      &1.09e-02&      &2.15e-03&        &8.02e-04&       \\
   40& 2.78e-02& 0.292&5.16e-03& 1.083&7.05e-04&   1.607&2.13e-04& 1.911\\
   80& 2.19e-02& 0.346&2.34e-03& 1.142&2.46e-04&   1.518&5.57e-05& 1.935\\
  160& 1.65e-02& 0.402&1.10e-03& 1.085&8.66e-05&   1.509&1.44e-05& 1.951\\
  320& 1.20e-02& 0.458&5.44e-04& 1.019&3.05e-05&   1.505&3.70e-06& 1.962\\
  640& 8.48e-03& 0.506&2.71e-04& 1.001&1.08e-05&   1.503&9.44e-07& 1.972\\
%1280& 5.82e-03& 0.544&1.358e-04& 1.001&3.800e-06&   1.502&2.383e-07& 1.986\\ 
\hline
\end{tabular}
\caption{Errors and convergence rates ($r_t$) for $\alpha=0.4$ and for different choices of $\gamma.$}
\label{table 1}
\end{center}
\end{table}

  \begin{table}
\begin{center}
\begin{tabular}{|c|cc|cc|cc|cc|}
\hline
$M$&\multicolumn{2}{c|}{$\gamma=1$ }
&\multicolumn{2}{c|}{$\gamma=2$}
&\multicolumn{2}{c|}{$\gamma=2.5$}&\multicolumn{2}{c|}{$\gamma=3$}\\
\hline
   20&   2.51e-02&        &   2.21e-03&        &   8.26e-04&        &   6.50e-04&        \\
   40&   1.50e-02&   0.748&   7.60e-04&   1.540&   2.12e-04&   1.964&   1.65e-04&   1.979\\
   80&   8.32e-03&   0.846&   2.66e-04&   1.513&   5.44e-05&   1.960&   4.17e-05&   1.982\\
  160&   4.53e-03&   0.879&   9.36e-05&   1.509&   1.46e-05&   1.897&   1.05e-05&   1.987\\
  320&   2.53e-03&   0.840&   3.30e-05&   1.505&   3.93e-06&   1.893&   2.65e-06&   1.991\\
  640&   1.47e-03&   0.779&   1.16e-05&   1.503&   1.06e-06&   1.892&   6.64e-07&   1.995\\
% 1280&  8.76e-04&   0.749&   4.11e-06&   1.502&   3.66e-07&   1.535&   3.66e-07&   8.611\\
\hline
\end{tabular}
\caption{Errors and convergence rates ($r_t$) for $\alpha=0.6$ and for different choices of $\gamma.$}
\label{table 2}
\end{center}
\end{table}

\begin{table}
\begin{center}
\begin{tabular}{|c|cc|cc|cc|cc|}
\hline
$M$&\multicolumn{2}{c|}{$\gamma=1$ }
&\multicolumn{2}{c|}{$\gamma=1.5$}
&\multicolumn{2}{c|}{$\gamma=2$}\\
\hline
   20&   9.7410e-03&         &   1.8813e-03&         &   5.4226e-04&         \\
   40&   4.2516e-03&   1.1961&   6.7825e-04&   1.4719&   1.3563e-04&   1.9993\\
   80&   2.0954e-03&   1.0207&   2.3983e-04&   1.4998&   3.4030e-05&   1.9948\\
  160&   1.0687e-03&   9.7134&   8.4803e-05&   1.4998&   8.5585e-06&   1.9914\\
  320&   5.3632e-04&   9.9475&   2.9991e-05&   1.4996&   2.1776e-06&   1.9746\\
\hline
\end{tabular}
\caption{Errors and convergence rates ($r_t$) for $\alpha=0.8$ and for different choices of $\gamma.$}
\label{table 3}
\end{center}
\end{table}

   For the graphical interpretation, we fixed  $N=160$  and $M=1200$, so the time error is dominant. Figure \ref{fig2}  shows how the error on uniform and nonuniform time meshes varies with $t$ for various  choices of $\alpha$, using a log scale. 
\begin{figure}
\begin{center}
\includegraphics[width=4.25cm, height=5cm]{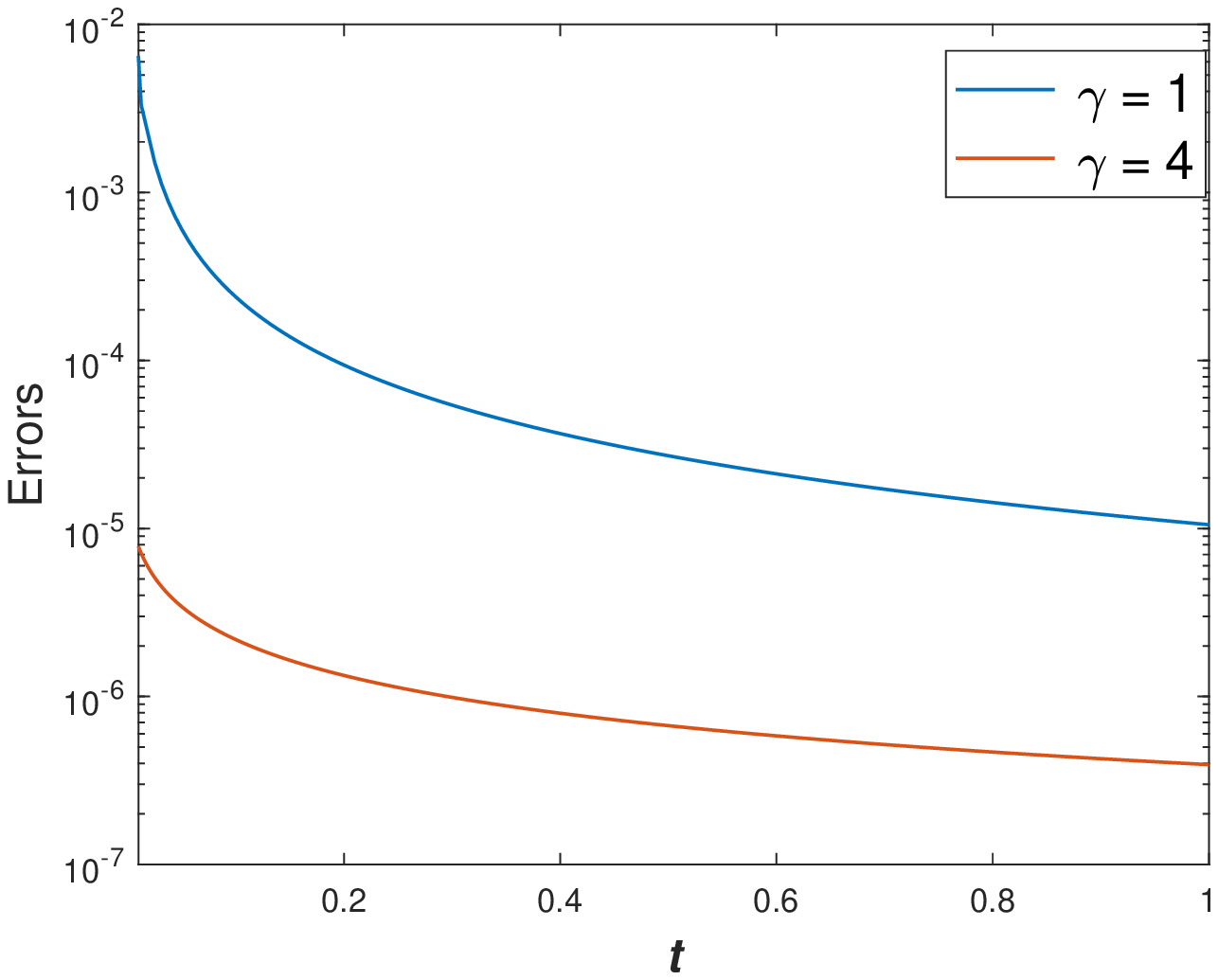}
\includegraphics[width=4.25cm, height=5cm]{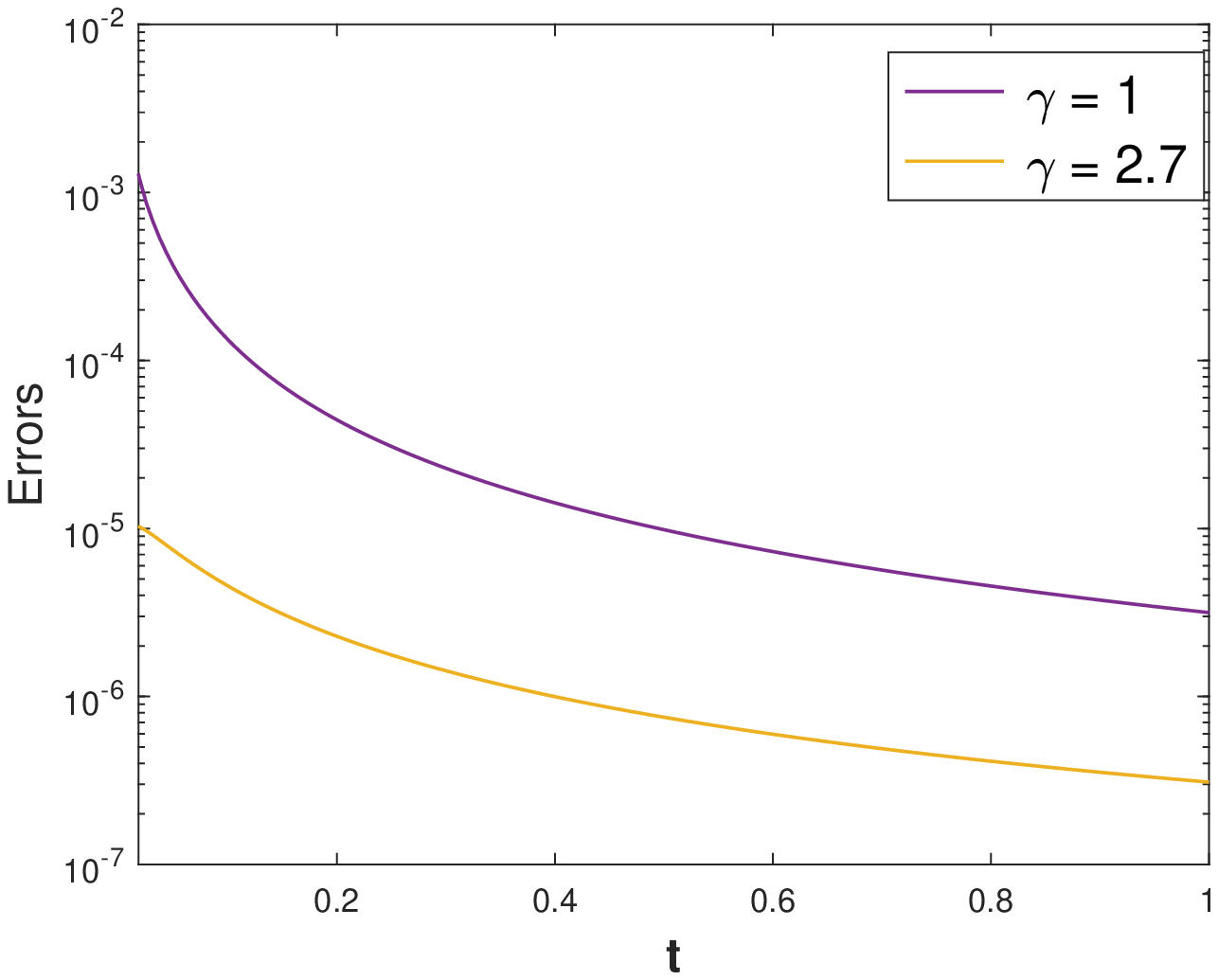}
\includegraphics[width=4.25cm, height=5cm]{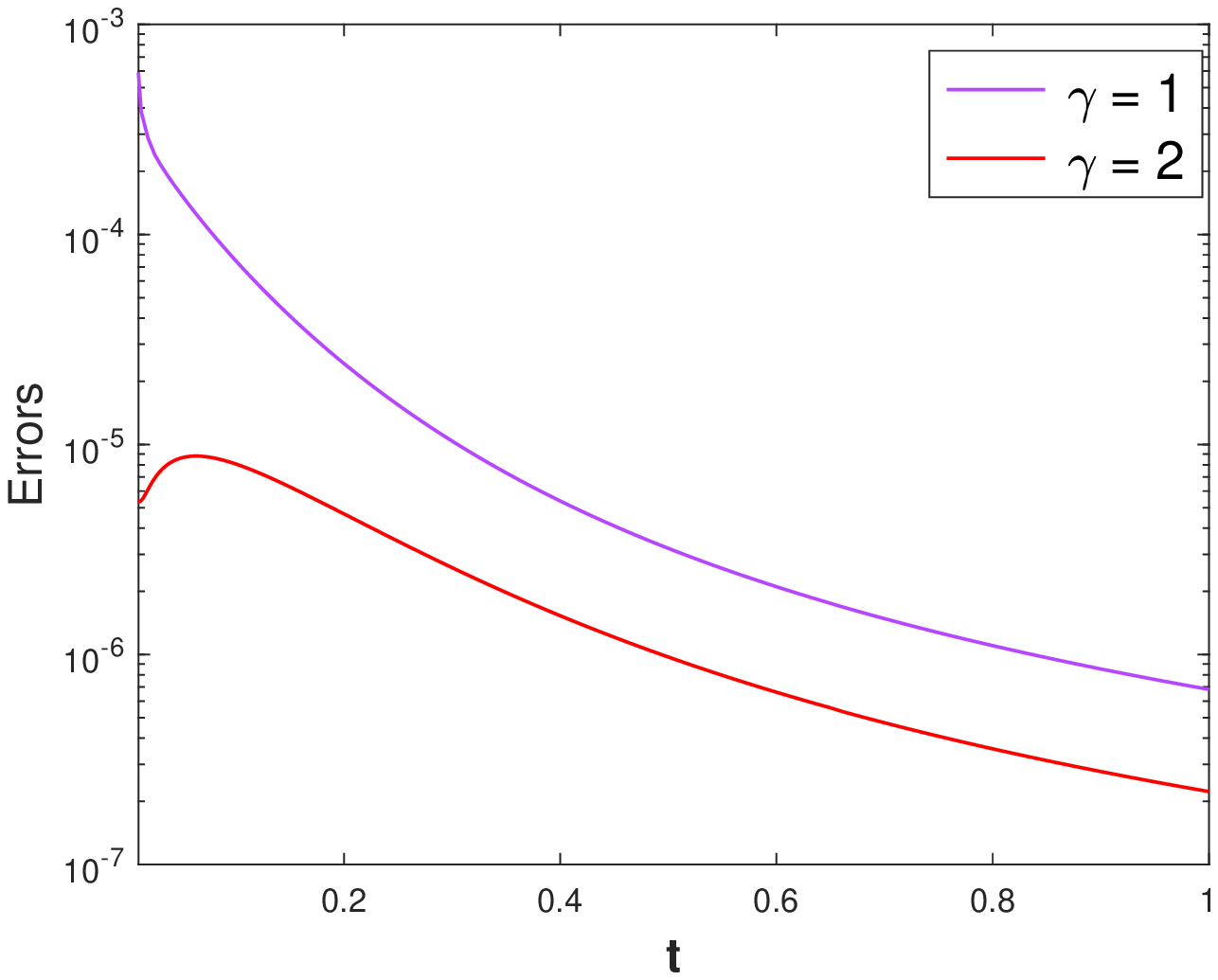}
\caption{The error $\|U_h^n-u(t_n)\|$ as a function of~$t_n$. The fractional exponent $\alpha=0.4$ in the left figure, while $\alpha=0.6$ in the middle one, and $\alpha=0.8$ in the right figure.}
\label{fig2}
\end{center}
\end{figure}

To demonstrate the $O(h^2)$-rates from the spatial discretization by Galerkin finite elements, the time mesh size is refined so that the errors in space  are dominant. The expected convergence orders are displayed in Table \ref{table 4} for $\alpha=0.3, 0.5,$ and $0.8.$  These results also illustrate that the condition $\sigma >(1-\alpha)/2$  in Theorem \ref{thm: CR} is not necessary. This condition holds true if $\alpha >2/3$ because $\sigma=\alpha^-/4,$ however an $O(h^2)$-rate  was observed despite  $\alpha$ not being greater than  $2/3.$  
\begin{table}
\begin{center}
\begin{tabular}{|c|cc|cc|cc|}
\hline
$M$&\multicolumn{2}{c|}{$\alpha=0.3$ }
&\multicolumn{2}{c|}{$\alpha=0.5$}
&\multicolumn{2}{c|}{$\alpha=0.8$}\\
\hline
   10& 8.4612e-04&         & 8.4612e-04&         &8.4612e-04&         \\
   20& 1.9932e-04&   2.0858& 1.9932e-04&   2.0858&1.9932e-04&   2.0858\\
   40& 4.8140e-05&   2.0498& 4.8140e-05&   2.0498&4.8140e-05&   2.0498\\
   80& 1.1795e-05&   2.0291& 1.1449e-05&   2.0720&1.1449e-05&   2.0720\\
  160& 3.0940e-06&   1.9306& 2.9063e-06&   1.9780&2.7481e-06&   2.0587\\
\hline
\end{tabular}
\caption{Errors and the spatial convergence rates $r_x$ for different values of $\alpha$.} 
%The number of time subintervals $N=1000$, $\gamma=4.5,2.$}
\label{table 4}
\end{center}
\end{table}

%\begin{figure}
%\begin{center}
%%\includegraphics[width=10cm, height=6cm]{Figure1Case3.eps}
%\caption{The error $|\|U^n-u(t_n)\||$ as a function of~$t_n$ for Example 3.}
%\label{fig3}
%\end{center}
%\end{figure}

\section{Concluding remarks} { An $L1$ time-stepping scheme for a time-fractional diffusion equation is developed. Over a sufficiently time-graded mesh, it is claimed that the proposed scheme is second-order accurate. Later on, our $L1$ scheme is combined with the standard Galerkin finite elements for the spatial discretization. The error analysis of the induced fully-discrete scheme is studied.  The delivered numerical tests confirmed that the achieved time-space convergence rates  are sharp, but the time-mesh exponent  $\gamma$ can be further relaxed. Due to several difficulties, improving the choice of $\gamma$ is beyond the scope of this work, it will be a subject of future research.}


\begin{thebibliography}{}

\bibitem{Alikhanov2015} A. A. Alikhanov, A new difference scheme for the time fractional diffusion equation, J.
Comput. Phys., 280 (2015), 424--438.

\bibitem{Brunner2004} H. Brunner,  Collocation Methods for Volterra Integral and Related Functional Equations Methods. Cambridge University Press, Cambridge (2004).

\bibitem{BrunnerPedasVainikko1999} H. Brunner, A. Pedas, and  Vainikko, The piecewise polynomial collocation method for weakly singular
Volterra integral equations. Math. Comp., 68 (1999), 1079--1095.

\bibitem{ChandlerGraham1988} G. A. Chandler and I. G. Graham, Product integration-collocation methods for noncompact
integral operator equations, Math. Comp., 50 (1988), 125--138.

\bibitem{ChenLiuAnhTurner2010} C. Chen, F. Liu, V. Anh, and I. Turner, Numerical schemes with high 
spatial accuracy for a variable-order anomalous subdiffusion equations, SIAM J. Sci.
Comput., 32 (2010), 1740--1760.

\bibitem{ChenStynes2019} H. Chen and M. Stynes, Blow-up of error estimates in time-fractional initial-boundary value problems, researchgate.net, 2019.


\bibitem{DixonMcKee1986} J. Dixon and S. McKee, Weakly singular Gronwall inequalities, ZAMM Z. Angew. Math. Mech., 66 (1986), 535--544.


\bibitem{GaoSun2011} G. Gao and Z. Sun, A compact finite difference scheme for the fractional 
sub-diffusion equations, J. Comput. Phys, 230 (2011), 586--565.

\bibitem{JiangMa2011} Y. Jiang and J. Ma, High-order finite element methods for time-fractional partial differential
equations, J. Comp. Appl. Math., 11 (2011), 3285--3290.

\bibitem{JinLazarovZhou} B. Jin, R. Lazarov, and Z. Zhou, An analysis of the $L1$ scheme for the subdiffusion
equation with nonsmooth data, IMA J. Numer. Anal., 36 (2016), 197--221.

\bibitem{JinLiZhou2019} B. Jin, B. Li, and Z. Zhou, Subdiffusion with a time-dependent coefficient:
analysis and numerical solution. Math. Comp., 88 (2019), 2157--2186.

\bibitem{Karaa2018} S. Karaa, Semidiscrete finite element analysis of time fractional parabolic problems: a
unified approach. SIAM J. Numer. Anal., 56 1673–1692, 2018.

\bibitem{KaraaMustaphaPani2018} S. Karaa, K. Mustapha, and A. K. Pani, Optimal error analysis of a FEM for fractional diffusion problems by energy arguments, J. Sci. Comput., 74 (2018), 519--535.

\bibitem{Kopteva2019} N. Kopteva, Error analysis of the $L1$ method on graded and uniform meshes for
a fractional-derivative problem in two and three dimensions. Math. Comp., 88 (2019), 2135--2155.

\bibitem{LiaoLiZhang2018} H-l Liao, D. Li, and J. Zhang, Sharp error estimate of the nonuniform $L1$ formula for linear reaction-subdiffusion equations, SIAM J. Numer. Anal., 56 (2018), 1112--1133.

\bibitem{LiaoMcLeanZhang2019} H.-l. Liao, W. McLean, and J. Zhang, A discrete Gronwall inequality with applications to numerical schemes 
   for subdiffusion problems, SIAM J. Numer. Anal., 57 (2019), 218--237. 
   
\bibitem{LinXu2007} Y. Lin and C. Xu, Finite difference/spectral approximations for the 
time-fractional diffusion equation, J. Comput. Phys., 225 (2007), 1552--1553.

   \bibitem{McLean2010} W. McLean, Regularity of solutions to a time-fractional diffusion equation, ANZIAM J., 52 (2010), 123--138.

\bibitem{McLeanMustapha2007} W. McLean and Mustapha, A second-order accurate numerical method for a fractional wave equation, Numer. Math., 105 (2007), 481--510.

\bibitem{McLeanMustapha2015} W. McLean and K. Mustapha, Time-stepping error bounds for fractional diffusion problems with non-smooth initial data, J. Comput. Phys., 293 (2015), 201--217.

\bibitem{McLeanMustaphaAliKnio2019o} W. McLean, K. Mustapha, R. Ali, and O. M. Knio, Well-posedness of time-fractional
advection-diffusion-reaction equations, Fract. Calc. Appl. Anal., 22 (2019), In press.

\bibitem{McLeanMustaphaAliKnio2019} W. McLean, K. Mustapha, R. Ali, and O. M. Knio, Regularity theory for time-fractional advection-diffusion-reaction equations, Computers \& Mathematics with Applications, (2019), In press.

\bibitem{McLeanThomeeWahlbin1996} W. McLean, V. Thom\'ee, and L. B. Wahlbin, Discretization with variable time steps of an evolution equation with a positive-type memory term, J. Comput. Appl. Math., 69 (1996), 49--69.

\bibitem{Mustapha2011} K. Mustapha, An implicit finite difference time-stepping method for a sub-diffusion equation, with spatial discretization by finite elements, IMA J. Numer. Anal., 31 (2011), 719--739.

\bibitem{Mustapha2013} K. Mustapha, A Superconvergent discontinuous Galerkin method for Volterra integro-differential equations, smooth and non-smooth kernels, Math. Comp., 82 (2013), 1987--2005.

\bibitem{Mustapha2015} K. Mustapha, Time-stepping discontinuous Galerkin methods for fractional diffusion problems.
Numer. Math., 130 (2015), 497--516.

\bibitem{MustaphaAbdallahFursti2014} K. Mustapha, B. Abdallah, and K. M. Furati, A discontinuous Petrov$-$Galerkin method for time-fractional diffusion equations, SIAM J. Numer. Anal., 52 (2014),  2512--2529.

\bibitem{MustaphaSchotzau2014} K.  Mustapha and D. Sch\"otzau, Well-posedness of $hp$-version discontinuous {G}alerkin methods for fractional diffusion wave equations, IMA J. Numer. Anal., 34 (2014), 1426--1446.

\bibitem{Pachpatte1987} B. G. Pachpatte, On the discrete generalisations of Gronwall’s inequality, J. Indian Math. Soc., 37 (1987), 147--156. 


\bibitem{StynesORiordanGracia2017}   M. Stynes, E. O'Riordan, and J. L. Gracia, Error analysis of a finite difference method
on graded meshes for a time-fractional diffusion equation, SIAM J. Numer. Anal., 55 (2017), 1057--1079.

\bibitem{Thomee2006} V. Thom\'ee, Galerkin Finite Element Methods for Parabolic Problems, second ed., Springer, 2006.

\bibitem{WangZhaoChenWeiTang2018}  F. Wang, Y. Zhao, C. Chen, Y. Wei, and Y. Tang, A novel high-order approximate scheme for two-dimensional 
time-fractional diffusion equations with variable coefficient, Computers \& Mathematics with Applications, 78 (2019), 1288--1301.

\bibitem{YanKhanFord2018} Y. Yan, M. Khan, and N. J. Ford, An analysis of the modified $L1$ scheme for time-fractional partial differential equations with nonsmooth data, SIAM J. Numer. Anal.,  56 (2018), 210--227.

\bibitem{ZhaoChenBuLiuTang2017} Y. Zhao, P. Chen, W. Bu, X. Liu, and Y. Tang, Two mixed finite element methods for time-fractional diffusion equations, J. Sci. Computing, 70 (2017), 407--428.

\bibitem{ZhaoZhangShiLiuTurner2016} Y. Zhao, Y. Zhang, D. Shi, F. Liu, and I. Turner, Superconvergence analysis of nonconforming finite element method for two-dimensional time fractional diffusion equations, Applied Mathematics Letters, 59 (2016), 38--47.

\end{thebibliography}
\end{document}